\documentclass[11pt,reqno]{amsart}

\usepackage[english]{babel}
\usepackage{csquotes}
\usepackage[style=numeric, 
            hyperref=auto,
            url=true,
            isbn=false,
            doi=false,
            block=none,
            maxnames=10,
            backend=biber]{biblatex}
\usepackage{amssymb}

\theoremstyle{plain}
\newtheorem{theorem}{Theorem}[section]
\newtheorem{conjecture}[theorem]{Conjecture}
\newtheorem{proposition}[theorem]{Proposition}
\theoremstyle{definition}
\newtheorem{definition}[theorem]{Definition}
\theoremstyle{remark}
\newtheorem{remark}[theorem]{Remark}
\newtheorem*{example}{Example}

\numberwithin{equation}{section}
\numberwithin{table}{section}

\newcommand{\C}{\mathbb C}
\newcommand{\Z}{\mathbb Z}
\newcommand{\F}{\mathbb F}
\newcommand{\Q}{\mathbb Q}
\newcommand{\PH}{{\mathbb H}}

\newcommand{\GL}{{\rm GL}}
\newcommand{\SL}{{\rm SL}}
\newcommand{\Sym}{{\rm Sym}}

\newcommand     {\qf}[3]        {[#1,#2,#3]}
\newcommand     {\tr}           {\sym{tr}}
\newcommand     {\sym}[1]       {\operatorname{#1}}
\newcommand     {\set}[1]       {{\def\st{\;:\;}\left\{#1\right\}}}

\begin{document}
\title[Computation of Siegel modular forms]{Computations of
  Vector-valued Siegel modular forms}
\author[Ghitza, Ryan, Sulon]{Alexandru Ghitza, Nathan C. Ryan, and
  David Sulon}

\thanks{
  Research of the first author was supported by an Early Career Researcher Grant
  from the University of Melbourne and a Discovery Grant from the Australian
Research Council.}  
\thanks{
  We thank G\"unter Harder for clarifying the case $\delta=2$
of Conjecture~\ref{conj:Harder} for us.  We thank Neil Dummigan for
many insightful comments regarding Conjecture~\ref{conj:Sym}.  We
thank Martin Raum for making available code for computing values of
symmetric square $L$-functions.}

\thispagestyle{empty}

\begin{abstract}
We carry out some computations of vector valued Siegel modular forms
of degree two, weight $(k,2)$ and level one.  Our approach is based on
Satoh's description of the module of vector-valued Siegel modular forms of
weight $(k, 2)$ and an explicit description of the Hecke action on Fourier
expansions.  We highlight three experimental results: 
(1) we identify a rational eigenform in a three dimensional space of cusp 
forms, 
(2) we observe that non-cuspidal eigenforms of level one are not always 
rational and 
(3) we verify a number of cases of conjectures about congruences between
classical modular forms and Siegel modular forms.  
\end{abstract}

\address{
{\parskip 0pt
Department of Mathematics and Statistics, University of Melbourne\endgraf
aghitza@alum.mit.edu\endgraf
\null
Department of Mathematics, Bucknell University\endgraf
nathan.ryan@bucknell.edu, david.sulon@bucknell.edu \endgraf
}
  }

\maketitle

\section{Introduction}

Computations of modular forms in general and Siegel modular forms in
particular are of great current interest.  Recent computations of Siegel
modular forms  on the paramodular group by Poor and Yuen~\cite{PoorYuen} have led to the
  careful formulation by Brumer and Kramer~\cite{BrumerKramer} of the
  Paramodular Conjecture, a natural generalization of
  Taniyama-Shimura.  Historically, computations of scalar-valued
  Siegel modular forms in the 1970s by Kurokawa~\cite{Kurokawa} led to the
  discovery of the Saito-Kurokawa lift, a construction whose
  generalizations are still studied.  In
  the 1990s, computations by Skoruppa~\cite{Skoruppa} revealed some
  striking properties that some Siegel modular forms possess (namely, there are
  rational eigenforms of weights 24 and 26 in level 1 that span a
  two-dimensional space of cusp forms).  These properties have yet to be
  explained.  This paper is in the same spirit.

We carry out the first systematic computations of spaces of 
vector-valued Siegel modular forms of degree two and of weight $(k,2)$.
We do this in Sage \cite{Sage} using a package co-authored by the
second author, Raum, Skoruppa and Tornar\'ia~\cite{smf-package}.  We observe
some new phenomena (see Propositions~\ref{prop:maeda} and
\ref{prop:eis}) and check that the eigenforms we compute satisfy the
Ramanujan-Petersson bound (see Proposition~\ref{prop:ram}).

We also verify two interesting conjectures on congruences:  
the first, due to Harder, has been previously verified
by Faber and van der Geer \cite{vanderGeer}; the other, due to Bergstr\"om,
Faber, van der Geer and Harder, has been previously verified by Dummigan
\cite{Dummigan2}.
Our approach to verifying these
conjectures is to compute Hecke eigenvalues of Siegel modular forms in as 
direct a manner as possible, using Satoh's
concrete description of Siegel modular forms of weight $(k, 2)$.  This
is a very different approach than the one taken by Faber and van der Geer and we
verify cases that they do not (and vice versa).  After
determining a basis of eigenforms for the space, we use explicit formulas
for the Hecke action on Fourier expansions to extract the Hecke eigenvalues.
Our main results in this direction are Theorem~\ref{thm:Harder}
and~\ref{thm:Sym}, which summarize the cases of the two conjectures 
that we have verified.

We do not describe the implementation of Siegel modular forms used to
carry out these computations but we refer the interested reader to
\cite{RaumRyanSkoruppaTornaria} for such a description.  The point of
this paper is that we were actually able to carry out such
computations, have made some new observations based on these computations 
and have made our data publicly available \cite{vv-smf-data}.

\section{Vector-valued Siegel modular forms of weight $(k, 2)$}
We recall the definition of a Siegel modular form of degree two.  We 
consider the \emph{full Siegel modular group} $\Gamma^{(2)}$
given by
\begin{equation*}
  \Gamma^{(2)}:=\textrm{Sp}(4,\Z)=\left\{
    M\in\textrm{M}(4,\Z)\colon {}^t M 
    \left(\begin{smallmatrix} & I_2 \\
    -I_2 & \end{smallmatrix}\right)
    M=
    \left(\begin{smallmatrix} & I_2 \\
      -I_2 & \end{smallmatrix}\right)\right\}.
\end{equation*}
Let
\begin{equation*}
  \PH^{(2)}
  :=
  \left\{Z\in \textrm{M}(2,\C)\;:\;{}^tZ=Z,\mathrm{Im}(Z)>0\right\}
\end{equation*}
be the Siegel upper half space of degree $2$.  For a nonnegative
integer $j$, the space $\C[X,Y]_j$ of homogeneous polynomials of 
degree $j$ has a $\GL_2$-action given by
\begin{equation}\label{eq:action}
  (A,p)\mapsto A\cdot p := p\bigl((X,Y)A\bigr).
\end{equation}

\begin{definition}
  \label{def:smfs}
  Let $k, j$ be nonnegative integers.
  A \emph{Siegel modular form} of degree 2 and weight $(k,j)$ is a
  complex analytic function $F\colon\PH^{(2)}\to \C[X,Y]_j$ such that
  \begin{equation*}
    F(gZ)
    :=
    F\bigl((AZ+B)(CZ+D)^{-1}\bigr) = \det(CZ+D)^k\,(CZ+D)\cdot F(Z)
  \end{equation*}
  for all
  $g=\left(\begin{smallmatrix}A&B\\C&D\end{smallmatrix}\right)\in\Gamma^{(2)}$.
\end{definition}
\begin{remark}\label{rmk:concrete}
The action in \eqref{eq:action} is a concrete realization of the symmetric power representation
$\mathrm{Sym}^j$ of $\GL_2$.  Definition~\ref{def:smfs} is a concrete description of Siegel
modular forms with values in the representation space
$\det^k\otimes\,\mathrm{Sym}^j$ of $\GL_2$.  We made these choices in
our implementation because it made the multiplication of vector-valued
Siegel modular forms easier to implement.
\end{remark}

The space of all such functions is denoted $M^{(2)}_{k,j}$,
where we suppress $j$ if it
is $0$. If $j$ is positive $F$ is called \emph{vector-valued},
otherwise it is called \emph{scalar-valued}.
We write $M_*^{(2)}:=\bigoplus_{k} M_k^{(2)}$ for the ring of (scalar-valued) Siegel modular
forms of degree $2$.

Let 
\[
Q := \set{f = [a,b,c] \st a,b,c\in\Z,\; b^2-4ac\leq 0,\; a\geq 0}
\] 
where $[a,b,c]$ corresponds to the quadratic form $aX^2+bXY+cY^2$.

A Siegel modular form $F$ has a Fourier expansion of the form
\[
  F(Z)
  = \sum_{f=\qf abc\in Q} C_F(f)\,e\left(a\tau+bz+c\tau^\prime\right).
\]
Here $Z:=\left(\begin{smallmatrix} \tau &
    z\\z&\tau^\prime\end{smallmatrix}\right)$ ($\tau,\tau^\prime\in\PH^{(1)}$ and
$z\in\C$), $e(x)=e^{2\pi i x}$,
the trace of a matrix $A$ is denoted by $\tr A$.  The form $F$ is called 
a \emph{cusp form} if its Fourier expansion is supported on
positive-definite elements of $Q$.  The subspace of cusp forms is denoted
$S^{(2)}_{k,j}$.

The ring of all vector-valued Siegel modular forms $\bigoplus_{k,j}
M_{k,j}^{(2)}$ is not finitely generated.  For this reason the
symmetric power $j$ is usually fixed.  The resulting module is
finitely generated over $M_*^{(2)}$.  We focus exclusively
on weight $(k, 2)$, where we have a very concrete
description of these spaces thanks to work of Satoh.

\subsection{Satoh's Theorem}\label{sec:satoh}

The Satoh bracket is a special case of
the general Rankin-Cohen bracket construction.  Satoh \cite{Satoh}
examined the case of weight $(k, 2)$.  Suppose $F\in M_{k}^{(2)}$ and
$G\in M_{k^\prime}^{(2)}$ are two scalar-valued Siegel modular forms.
We define the \emph{Satoh bracket} by
\begin{align*}
  [F, G]_{2} & = \frac{1}{2 \pi i} \left(\frac{1}{k} G\, \partial_Z F
  - \frac{1}{k^\prime} F\, \partial_Z G\right) \in M_{k+k^\prime,2}^{(2)}\text{,}
\end{align*}
where $\partial_Z = \left(\begin{smallmatrix} \partial_{Z_{11}} &
    1/2\, \partial_{Z_{12}} \\ 1/2\, \partial_{Z_{12}}
    & \partial_{Z_{22}}\end{smallmatrix}\right)$.

In the same paper, Satoh showed that $\bigoplus_k M_{k,2}^{(2)}$ is
generated by elements all of which can be expressed in terms of Satoh 
brackets.  More precisely, he showed that
\begin{align*} 
  M_{k,2}^{(2)}
=&
  [E_4, E_6]_2 \cdot M_{k-10}^{(2)} \oplus
  [E_4, \chi_{10}]_2 \cdot M_{k-14}^{(2)} \oplus
\\&
  [E_4, \chi_{12}]_2 \cdot M_{k-16}^{(2)} \oplus
  [E_6, \chi_{10}]_2 \cdot \C[E_6, \chi_{10}, \chi_{12}]_{k-16} \oplus
\\&
  [E_6, \chi_{12}]_2 \cdot \C[E_6, \chi_{10}, \chi_{12}]_{k-18} \oplus
  [\chi_{10}, \chi_{12}]_2 \cdot \C[\chi_{10}, \chi_{12}]_{k-22}
\text{.}
\end{align*}
Here the forms $E_4,E_6,\chi_{10},\chi_{12}$ are the generators of the
ring of scalar-valued Siegel modular forms described by Igusa
\cite{Igusa}.  By $\C[A_1,\dots,A_n]_k$ we mean the module of weight $k$ modular
  forms that can be expressed in terms of generators $A_1,\dots,A_n$.

  A basis for the space $M_{k,2}^{(2)}$ was computed via a Sage \cite{Sage}
implementation in \cite{smf-package} of an algorithm found in
\cite{RaumRyanSkoruppaTornaria}.  

\subsection{Hecke Operators}
As our interest is in computing Hecke eigenforms, we need to describe
how one computes the Hecke action.  We give formulas for the image of a Siegel
modular form of weight $(k, 2)$ under the operator
$T(p^\delta)$.  The Hecke operators are multiplicative and so it
suffices to understand the image for these operators.  The formulas
can be found in \cite{Ibukiyama} but we present them here for completeness.

Let $F$ be a Siegel modular form as above and let the image of $F$
under $T(p^\delta)$ have coefficients $C^\prime([a,b,c])$.  Then 
\begin{multline}\label{eq:hecke}
C^\prime([a,b,c])=
\sum_{\alpha+\beta+\gamma=\delta}p^{\beta k+\gamma(2k-1)}\times\\
\sum_{\substack{U\in R(p^\beta)\\a_U \equiv
    0\,(p^{\beta+\gamma})\\ b_U \equiv c_U\equiv 0\,(p^{\gamma})}} (d_{0,\beta}U)\cdot C\left(p^{\alpha}\left[\frac{a_U}{p^{\beta+\gamma}},\frac{b_U}{p^\gamma},\frac{c_U}{p^{\gamma-\beta}}\right]\right)
\end{multline}
where 
\begin{itemize}
  \item
$R(p^\beta)$ is a complete set of representatives for
$\SL(2,\Z)/\Gamma^{(1)}_0(p^\beta)$ where $\Gamma^{(1)}_0(p^\beta)$ is
the congruence subgroup of $\SL(2,Z)$ of level $p^\beta$; 
\item for $f=[a,b,c]$,
$[a_U,b_U,c_U]=f_U:=f\left((X,Y){}^tU\right)$; 
\item $d_{0,\beta} =
\left(\begin{smallmatrix}1 & \\ & p^\beta\end{smallmatrix}\right)$;
\item the $\cdot$ is given by the action defined in \eqref{eq:action}.
\end{itemize}
We denote the Hecke eigenvalue of a Siegel modular form $F$ under the
operator $T(p^\delta)$ by $\lambda_{p^\delta}(F)$.
If the space $S_{k,2}^{(2)}$ has dimension $d$, the
Hecke eigenvalues of $F$ are algebraic numbers of degree at most $d$.  The
field that contains the Hecke eigenvalues of $F$ is denoted $\Q_F$.

\subsection{Computing Hecke eigenforms}\label{sec:eigenforms}

Fix a space of Siegel modular forms of weight $(k, 2)$ with basis
$\{F_1,\dots,F_n\}$, obtained as algebraic combinations of
the Igusa generators and Satoh brackets.  Because the Hecke operators are a commuting
family of linear operators, there is a basis $\{G_1,\dots,G_n\}$ for
the space consisting entirely of simultaneous eigenforms.

The forms $G_i$ are determined computationally as follows.  First,
determine the matrix representation for the Hecke operator $T(2)$ by
computing the image under $T(2)$ of each basis element $F_i$.  Build a
matrix $N$ that is invertible and whose $j$th row consists of coefficients of
$F_j$ at certain indices $Q_1,\dots,Q_n$.  To ensure that $N$ is
invertible we pick the indices one at a time, making sure that each
choice of index $Q_i$ increases the rank of $N$.  We then
construct a matrix $M$ whose $j$th row consists of coefficients of the
image of $F_j$ under $T(2)$ indexed by $Q_1,\dots,Q_n$.  Then the
matrix representation of $T(2)$ is $M N^{-1}$.  

We compute the Hecke eigenforms using $T(2)$ and we express them in terms
of the basis $\{F_1,\dots,F_n\}$.  We compute the Hecke eigenvalues
$\lambda_{p^\delta}$ by computing these expressions to high
precision and then computing their image under the
Hecke operator $T(p^\delta)$ as in \eqref{eq:hecke}.

\subsection{Hecke eigenvalues, Satake parameters and symmetric polynomials}
Fix a prime $p$.
The Satake isomorphism $\Omega$ is a map between the local-at-$p$ Hecke
algebra $\mathcal{H}_p$ associated to $\Gamma$ and a polynomial ring
$\Q[x_0,x_1,x_2]^{W_2}$ invariant
under the action of the Weyl group.  A matrix representation of $\Omega$
can be found in \cite{Krieg,RyanShemanske}.  We summarize the relevant
results here.  

Consider 
\begin{align*}
  T(p)&=\Gamma\begin{pmatrix}1\\ &1\\ &&p\\ &&&p\end{pmatrix}\Gamma, &
  T_0(p^2)&=\Gamma\begin{pmatrix}1\\ &1\\ &&p^2\\ &&&p^2\end{pmatrix}\Gamma\\
  T_1(p^2)&=\Gamma\begin{pmatrix}1\\ &p\\ &&p^2\\
    &&&p\end{pmatrix}\Gamma, &
  T_2(p^2)&=\Gamma\begin{pmatrix}p\\ &p\\ &&p\\ &&&p\end{pmatrix}\Gamma.
\end{align*}
The images under $\Omega$ of these
operators are:
\begin{align*}
\Omega(T(p)) &= x_0x_1x_2+x_0x_1+x_0x_2+x_0\\
\Omega(T_0(p^2)) &= \tfrac{2p-2}{p}
\phi_2+\tfrac{p-1}{p}\phi_1+\phi_0\\
\Omega(T_1(p^2)) &= \tfrac{p^2-1}{p^3}\phi_2 + \tfrac{1}{p}\phi_1\\
\Omega(T_2(p^2) &= \tfrac{1}{p^3} \phi_2
\end{align*}
where 
\begin{align*}
\phi_0 &= x_0^2x_1^2x_2^2+x_0^2x_1^2+x_0^2x_2^2+x_0^2\\
\phi_1 &= x_0^2x_1^2x_2+x_0^2x_1x_2^2+x_0^2x_1+x_0^2x_2\\
\phi_2 &= x_0^2x_1x_2.
\end{align*}
Fix a Siegel Hecke eigenform $F$.  It can be shown
\cite[p. 165]{Andrianov} that for any $p$ there exists a triple
$(\alpha^F_{0,p},\alpha^F_{1,p},\alpha^F_{2,p})\in (\C^\times)^3/W_2$
with the property that $\Omega(T)|_{x_i\leftarrow \alpha^F_{i,p}}=\lambda_T(F)$, the
eigenvalue of $F$ with respect to the Hecke operator $T$.  The numbers
$\alpha$ are called the \emph{Satake parameters} of $F$ at $p$.

We will make use of the following way of expressing $T(p)^2$ in terms of the
operators $T_i(p^2)$:
\begin{theorem}[\cite{vanderGeer}]\label{thm:hecke_relation}
  $T(p)^2=T_0(p^2)+(p+1)T_1(p^2)+(p^2+1)(p+1)T_2(p^2)$.
\end{theorem}

\section{Computational and Experimental Results}\label{sec:computations}

We carry out the computations of particular eigenforms in the
following way.  Satoh's theorem as described above in
Section~\ref{sec:satoh} gives a recipe for computing a basis of
vector-valued Siegel modular forms of weight $(k,2)$.  In particular,
following \cite{Skoruppa} we can compute the Fourier expansions of the Igusa
generators $E_4,E_6,\chi_{10},\chi_{12}$.  This is done via an
explicit map from elliptic modular forms to Siegel modular forms.  For
each generator we easily computed the part of its Fourier expansion
which is supported on positive definite quadratic forms up to discriminant
3000 and on singular quadratic forms $[0,0,c]$ where $0\leq c\leq 750$.  

Using these four Igusa generators we determine a basis for the space
of weight $(k,2)$ as prescribed by Satoh's theorem:  we compute a
basis for the modules $M_{k-10}$, $M_{k-14}$, $M_{k-16}$ and 
$\C[E_6,\chi_{10}, \chi_{12}]_{k-16}$, $\C[E_6, \chi_{10},
  \chi_{12}]_{k-18}$,
$\C[\chi_{10}, \chi_{12}]_{k-22}$.  We then form a basis for the space
of weight $(k,2)$ by multiplying each basis above by the
Satoh bracket that corresponds to it in Satoh's theorem and end up with 
eigenforms following the procedure
described in Section~\ref{sec:eigenforms}.  One might pause at the
idea of multiplying vector-valued Siegel modular forms but this is
precisely the reason why we defined the coefficients of vector-valued
Siegel modular forms to be homogeneous polynomials (see Remark~\ref{rmk:concrete}). 

For example, we find that a basis for the space of weight $(16,2)$ is
given by
\begin{align*}
G_1 &= E_6[E_4,E_6]_2-\tfrac{173820100608}{1557539}[E_4,\chi_{12}]_2 + \tfrac{1800409600}{1557539}[E_6,\chi_{10}]_2\\
G_2 &= [E_4,\chi_{12}]_2+ (\tfrac{5}{8064}\alpha +
\tfrac{755}{42})[E_6,\chi_{10}]_2\\
\overline{G}_2 &= [E_4,\chi_{12}]_2+ (\tfrac{5}{8064}\bar\alpha + \tfrac{755}{42})[E_6,\chi_{10}]_2\\
\end{align*}
where $\alpha$ is a root of $x^2 + 58752x + 858931200$ and
$\bar\alpha$ is its conjugate.  The form $G_1$
is non-cuspidal (probably Eisenstein) and the other two forms in 
the basis are cuspidal.

The Hecke eigenvalues that appear in the space $S_{k,2}^{(2)}$ tend to have
  the largest possible degree, namely the dimension $d$ of the space.
  For example the forms $G_2$ and $\overline{G}_2$ above have Hecke
  eigenvalues in a quadratic field.  There
  are however (very surprisingly!) counterexamples to this; this is 
  analogous to what
  happens in the scalar-valued spaces $S^{(2)}_{24}$ and $S^{(2)}_{26}$, 
  see \cite{Skoruppa}, but in stark contrast to the situation in degree
  one (as predicted by Maeda's conjecture).

  Consider the weight $(20, 2)$.
  Let $K=\Q[\alpha]$ be the quadratic number field with
  minimal polynomial $x^2 - 780288x + 121332695040$, and let $\bar\alpha$
  denote the conjugate of $\alpha$ in $K$.  The space $S_{20,2}^{(2)}$ is
  three-dimensional, and a basis of Hecke eigenforms is given by:
\begin{align*}
H_1&=\chi_{10}[E_4,E_6]_2-\tfrac{5}{14}E_6[E_4,\chi_{10}]\\
H_2&=\chi_{10}[E_4,E_6]_2+(\tfrac{25}{12241152}\alpha -
\tfrac{7685}{15939})E_6[E_4,\chi_{10}]_2+\\
&\phantom{XXXXXXXXXXXXXXXXXXXXX}(\tfrac{-1}{364320}\alpha +
\tfrac{674}{759})E_4[E_4,\chi_{12}]_2\\
H_3&=\chi_{10}[E_4,E_6]_2+(\tfrac{25}{12241152}\bar{\alpha} -
\tfrac{7685}{15939})E_6[E_4,\chi_{10}]_2+\\
&\phantom{XXXXXXXXXXXXXXXXXXXXX} (\tfrac{-1}{364320}\bar{\alpha} +
\tfrac{674}{759})E_4[E_4,\chi_{12}]_2.
\end{align*}

Note that the first has rational eigenvalues, and the second and third have
conjugate quadratic eigenvalues.  We checked that each of these forms
satisfies the Ramanujan-Petersson conjecture.

\begin{proposition}\label{prop:ram}
  For all $k$ satisfying $14\leq k\leq 30$, the Hecke eigenforms in
  $S^{(2)}_{k,2}$ satisfy the Ramanujan-Petersson conjecture at $p=2,3,5$.  
  More precisely, let $F\in S^{(2)}_{k,2}$ be a Hecke eigenform with
  eigenvalues $\lambda_p$, $\lambda_{p^2}$ and consider the polynomial
  \begin{equation*}
    X^4 - \lambda_p X^3
    + \left(\lambda_p^2-\lambda_{p^2}-p^{2k-2}\right)X^2
    - p^{2k-1}\lambda_p X
    + p^{4k-2}.
  \end{equation*}
  Then all roots $z\in\C$ of this polynomial satisfy
  $|z|=p^{(2k+j-3)/2}$.
\end{proposition}

We also looked at the level 1
elliptic modular forms with Hecke eigenvalues in quadratic
fields and these fields are different than $K$.  This indicates that
$H_2$ and $H_3$ are unlikely to be lifts.
Therefore the naive generalization
of Maeda's conjecture does not hold for $S_{k,2}^{(2)}$.  We remark that in
all other weights for which we have carried out computations, the naive
generalization does indeed hold:
\begin{proposition}\label{prop:maeda}  
Let $k\in\{14,16,18, 22,24,26,28,30\}$.  Then the
characteristic polynomial of the Hecke operator $T(2)$ acting on
$S_{k,2}$ is irreducible
over $\Q$.  If $k=20$, the characteristic polynomial of the Hecke
operator $T(2)$ decomposes over $\Q$ into a linear factor and a
quadratic factor.
\end{proposition}

We also make note of another interesting computational phenomenon
that merits further investigation.  It is interesting to us because an
analogous phenomenon does not happen in the scalar-valued case.  In the
scalar-valued case of level 1, there are four kinds of modular forms
of even weight: Eisenstein series, Klingen-Eisenstein series,
Saito-Kurokawa lifts and cusp forms that are not lifts.  The first two
are not cuspidal and always have rational coefficients.  Compare this
fact to the following proposition:
\begin{proposition}\label{prop:eis}
Let $k\in\{ 22,26,28,30\}$.  The space of modular forms of weight
$(k,2)$ that are not cusp forms is two dimensional but consists of a
single Galois orbit.
\end{proposition}

The data that make up the proofs of these propositions can be found at \cite{vv-smf-data}.

\section{Verification of some conjectural congruences}
The most famous modular form is arguably
\begin{equation*}
  \Delta(q)=q-24q^2+252q^3-1472q^4+4830q^5-\ldots=
  \sum_{n=1}^\infty \tau(n)q^n.
\end{equation*}

Ramanujan discovered a number of congruences involving the coefficients
$\tau(n)$, among which is
\begin{equation*}
  \tau(p)\equiv p^{11}+1\pmod{691}\quad\text{for all primes }p.
\end{equation*}

This is part of a more general phenomenon: if a prime $\ell\geq k-1$ divides
the numerator of the zeta-value $\zeta(-k+1)$ (equivalently, the numerator 
of the Bernoulli
number $B_k$), then the constant term of the Eisenstein series $E_k$ is zero
modulo $\ell$.  This can be interpreted to say that there is a congruence 
mod $\ell$ between this Eisenstein
series and some cuspidal eigenform of weight $k$.

As explained in \cite{vanderGeer} and \cite{Harder}, Deligne's work on
attaching families of $\ell$-adic Galois representations to Hecke eigenforms
of degree one allows us to interpret Ramanujan's congruence as taking place
between traces of Frobenius acting on cohomology spaces of local systems.

A more recent development is the construction (initiated by Laumon
\cite{Laumon} and Taylor \cite{Taylor}
and completed by Weissauer \cite{Weissauer}) of families of four-dimensional $\ell$-adic Galois
representations attached to Siegel modular eigenforms of degree two.  It is
then natural to ask about generalizations of Ramanujan's congruence to this
setting.  Building on his study of Eisenstein cohomology for arithmetic
groups, Harder stated a conjecture \cite{Harder} regarding congruences between
classical (degree one) eigenforms and Siegel eigenforms of degree two.  This
statement, which appears below as Conjecture~\ref{conj:Harder}, was verified 
in a number of cases by Faber and van der Geer \cite{vanderGeer}, who 
calculated the number of points on the relevant moduli
spaces over finite fields and related them to the Hecke eigenvalues of Siegel
modular forms.  (This relation was stated as a conjecture in
\cite{vanderGeer}, but it has since been proved by van der Geer and Weissauer
in most cases.  We refer the interested reader to
\cite{BergstromFabervanderGeer} for more details.)  

Bergstr\"om, Faber and van der Geer have extended this point counting
approach to Siegel modular forms of degree three in
\cite{BergstromFabervanderGeer}.  At the same time, they formulated
another conjectural congruence relating Siegel modular forms of degree
two and classical eigenforms, this time via critical values of symmetric
square $L$-functions.  We state this below as Conjecture~\ref{conj:Sym}.
It generalizes a result of Katsurada and Mizumoto for scalar-valued
Siegel modular forms (see~\cite{KatsuradaMizumoto}), and a number of 
vector-valued cases have been
proved by Dummigan in~\cite[Proposition 4.4]{Dummigan2}.

We verify the conjectures by using the data we collected in
Section~\ref{sec:computations} and some custom Sage code that
computes critical values of $L$-functions.  

\subsection{Notation}

The conjectures appear in different forms in \cite{Harder}, \cite{vanderGeer}
and \cite{BergstromFabervanderGeer}.  We follow the approach of 
\cite{BergstromFabervanderGeer} and
adapt it to our notation and the quantities that we compute.

We denote the space of cusp
forms of weight $r$ with respect to the group $\Gamma^{(1)}=\SL(2,\Z)$ by
$S_r^{(1)}$.  Suppose $f\in
S_r^{(1)}$ is a Hecke eigenform; we denote its Hecke eigenvalue with
respect to the operator $T(n)$ by $a_n=a_n(f)$.  The spaces $S_r^{(1)}$ are
finite dimensional, say of dimension $d$; according to Maeda's
conjecture~\cite{Maeda}, 
the eigenvalues $a_n$ are algebraic numbers of degree $d$.  The number
field that contains the coefficients is denoted $\Q_f$.

Fix a prime $p$.  For an eigenform $f\in S_r^{(1)}$, let $\alpha_0$ and 
$\alpha_1$ denote the Satake parameters at $p$ and define
\begin{equation*}
  \mu_{p^\delta}(f) = \alpha_0^\delta + \alpha_0^\delta\alpha_1^\delta\quad
  \text{for }\delta\geq 1.
\end{equation*}

Similarly, for a Siegel eigenform $F\in S_{k,j}^{(2)}$, let $\alpha_0$,
$\alpha_1$, $\alpha_2$ be the Satake parameters at $p$ and define
\begin{equation*}
  \mu_{p^\delta}(F) = \alpha_0^\delta + \alpha_0^\delta\alpha_1^\delta
  + \alpha_0^\delta\alpha_2^\delta + \alpha_0^\delta\alpha_1^\delta\alpha_2^\delta\quad
  \text{for }\delta\geq 1.
\end{equation*}
\subsection{$L$-functions of modular forms}

Let $f(q) = \sum a_n q^n \in S^{(1)}_r$ be an eigenform and consider
\begin{equation*}
L(f, s) = \sum_{n = 1}^\infty \frac{a_n}{n^s}=\prod_p
\left(1-a_pp^{-s}+p^{r-1-2s}\right)^{-1}.
\end{equation*}
After introducing the factor at infinity $L_\infty(f,
s)=\Gamma(s)/(2\pi)^s$, the completed $L$-function
\begin{equation}\label{eqn:lfcn}
\Lambda(f,s) = \frac{\Gamma(s)}{(2\pi)^s} L(f,s) = \int_0^\infty f(i
y)y^{s - 1} dy
\end{equation}
has holomorphic continuation to $\C$ and satisfies the functional
equation
\begin{equation*}
  \Lambda(f, s)= (-1)^{r/2} \Lambda(f, r-s).
\end{equation*}
Its \emph{critical values} occur at $1\leq t\leq r-1$ (of course, the functional
equation implies that it suffices to consider half of this interval).

Manin and Vishik proved that there exist real numbers $\omega_+(f),
\omega_-(f)$, called \emph{periods} of $f$, such that the ratio of the
critical values of $\Lambda(f, s)$ and the periods is algebraic.  More
precisely, define the \emph{algebraic critical values}
\begin{equation*}
  \widetilde{\Lambda}(f, t)=\begin{cases}
    \Lambda(f, t)/\omega_+(f) & \text{if $t$ is even}\\
    \Lambda(f, t)/\omega_-(f) & \text{if $t$ is odd}.
  \end{cases}
\end{equation*}
\begin{theorem}[Manin-Vishik\cite{Manin}]\label{th:manin_omega}
  If $f \in S_r^{(1)}$ is an eigenform and $t$ is an integer satisfying 
  $1\leq t\leq r-1$, then $\widetilde{\Lambda}(f, t)\in\Q_f$.
\end{theorem}

As explained in~\cite{Harder}, the denominators of certain classes in the
cohomology groups of local systems on the moduli space of abelian surfaces
should be expressed in terms of critical values $\widetilde{\Lambda}(f,
t)$.  Harder conjectured that the appearance of certain large primes
in these denominators should imply the existence of congruences between
eigenvalues of forms of degree one and two.

\begin{conjecture}[Harder]\label{conj:Harder}
  Let $f\in S_r^{(1)}$ be a Hecke eigenform with coefficient field $\Q_f$
  and let $\ell$ be an ordinary prime in $\Q_f$ (i.e. such that the $\ell$-th
  Hecke eigenvalue of $f$ is not divisible by $\ell$).  Suppose
  $s\in\mathbb{N}$
  is such that $\ell^s$ divides the algebraic critical value $\widetilde\Lambda(f, t)$.
  Then there exists a Hecke eigenform $F\in S_{k, j}^{(2)}$, where
  $k=r-t+2$, $j=2t-r-2$, such that
  \begin{equation*}
    \mu_{p^\delta}(F)\equiv \mu_{p^\delta}(f) + p^{\delta(k+j-1)}
    + p^{\delta(k-2)}\pmod{\ell^s}
  \end{equation*}
  for all prime powers $p^\delta$.
\end{conjecture}

\subsubsection{Computation of the quantities $\mu_{p^\delta}$}
The first step in our numerical verification of the conjecture is to
compute, as described in Section~\ref{sec:computations}, the Hecke 
eigenforms in various weights and their
corresponding Hecke eigenvalues $\lambda_p(F)$ and $\lambda_{p^2}(F)$.
Once we have those, the second step is to relate the Hecke eigenvalues
to the values $\mu_{p^\delta}(F)$.  We do this by relating the
polynomials that define $\mu_{p^\delta}(F)$ to the expressions of
$\lambda_p(F)$ in terms of the Satake parameters
$\alpha_0,\alpha_1,\alpha_2$. 

\begin{example}
Let $\lambda_i(p^2)=\Omega(T_i(p^2))|_{x_i\leftarrow \alpha_i}$.  Note
$\lambda_2(p^2)=p^{2k+j-6}$ from the definition of the slash
operator.  Then we have the equations
\begin{align*}
\lambda_p^2 &= \lambda_0(p^2)+(p+1)\lambda_1(p^2)+(p^2+1)(p+1)\lambda_2(p^2)\\
\lambda_{p^2} &= \lambda_0(p^2)+\lambda_1(p^2)+\lambda_2(p^2)
\end{align*}
where the first equation comes from Theorem~\ref{thm:hecke_relation}.

We compute $\lambda_p,\lambda_{p^2}$ and know $\lambda_2(p^2)$.  This
allows us to solve for $\lambda_0(p^2)$ and $\lambda_1(p^2)$.
Then we observe
\begin{align*}
\mu_p &=\lambda_p\\
\mu_{p^2}&=2\lambda_{p^2}-\lambda_p^2+2p^{2k+j-4}\\
\mu_{p^3}&=\left(3\lambda_{p^2}-2\lambda_p^2+3(p+1)p^{2k+j-4}\right)\lambda_p.
\end{align*}
\end{example}




\subsubsection{Computation of the congruence primes $\ell$}
We consider the ratios
\begin{equation*}
\Lambda(f,2) : \Lambda(f,4) : \dots : \Lambda(f,r - 2) 
\text{ and } \Lambda(f,1) : \Lambda(f,3) : \dots : \Lambda(f,r - 1) 
\end{equation*}
which by Theorem~\ref{th:manin_omega} are in $\Q_f$.

We compute these ratios of critical values as floating point numbers
in Sage \cite{Sage}.  This is done via an implementation of
\eqref{eqn:lfcn} due to Dokchitser~\cite{Dokchitser}.  We take these 
floating point numbers and find their
minimal polynomial using fplll~\cite{fplll}, an implementation of the
LLL lattice
reduction algorithm wrapped in Sage.  We
provide an example to illustrate our process and summarize our
computations in Table~\ref{tbl:Harder}.

\begin{example}
  Let $g \in S_{32}^{(1)}$, a two-dimensional space of cusp forms.  
  We will calculate the ratio of critical values
\begin{equation*}
\Lambda(g,1) : \Lambda(g,3) : \Lambda(g,5): \Lambda(g,7):
\Lambda(g,9):\Lambda(g,11):\Lambda(g,13):\Lambda(g,15).
\end{equation*}
As we are interested only in the ratio, we compute
\begin{multline*}
\tfrac{\Lambda(g,1)}{\Lambda(g,1)} : \tfrac{\Lambda(g,3) }{\Lambda(g,1)}
: \tfrac{\Lambda(g,5) }{\Lambda(g,1)}: \tfrac{\Lambda(g,7)
}{\Lambda(g,1)}:\tfrac{\Lambda(g,9) }{\Lambda(g,1)}:\tfrac{\Lambda(g,11)
}{\Lambda(g,1)}:\tfrac{\Lambda(g,13)
}{\Lambda(g,1)}:\tfrac{\Lambda(g,15) }{\Lambda(g,1)}=\\
1: 0.045375\dots:0.002369\dots:0.000143\dots: 0.000010\dots\\
8.65221\dots\times 10^{-7}:
8.50052\dots\times 10^{-8}: 9.23745\times 10^{-9}
\end{multline*}
using a Sage implementation of $\Lambda(g,s)$.  Then we find the
minimal polynomial of each ratio; e.g.,
$\tfrac{\Lambda(3)}{\Lambda(1)}$ has minimal polynomial
\begin{equation*}
1254224510x^2 - 471820065x + 18826702.
\end{equation*}
\end{example}

\subsubsection{Checking ordinarity of the primes $\ell$}

We use a simple algorithm that is very fast but uses large amounts of
storage.  Let $d$ be the dimension of $S_r^{(1)}$, and suppose we want
to check that $\ell$ is ordinary for all Hecke eigenforms in
$S_r^{(1)}$.  We proceed as follows: 
(a) compute the Victor Miller basis for $S_r^{(1)}$ to a precision of 
about $d\ell$ coefficients; 
(b) compute the matrix of the Hecke operator $T_\ell$ acting on this 
basis; 
(c) reduce the matrix modulo $\ell$ and check whether it is invertible.

This allowed us to verify that most primes $\ell$ appearing
in Table~\ref{tbl:Harder} are ordinary.
The current understanding of the distribution of non-ordinary primes 
is rather limited, but numerical evidence seems to indicate that they 
are very rare in level one, so it would be surprising to find a prime 
$\ell>r$ that divides an algebraic critical value and is non-ordinary.

\subsubsection{Verification of the congruences}

We observe that there are two cases: when
$\delta=1$ and when $\delta >  1$.  The difference is that in the
first case the congruence reduces to a congruence on the Hecke
eigenvalue $\lambda_p$ while in the second case we require both $\lambda_p$
and $\lambda_{p^2}$.  The effect of this difference is that we can
verify many more congruences when $\delta=1$ than when $\delta >
1$; this is due to the number of coefficients needed to compute
$\lambda_p$ as compared to the number needed to compute $\lambda_p$
and $\lambda_{p^2}$.

The way the actual verification works is essentially the same starting
from the point where $\mu_{p^\delta}(F)$ and $\mu_{p^\delta}(f)$ have
been computed.  Each side of the congruence mod $\ell^s$ is an 
algebraic number in $\Q_F$ and $\Q_f$ respectively.  In the cases we 
have considered, the exponent $s$ appearing in the conjecture was always
$1$.
We compute the minimal polynomial $m(x)$ of the coefficient $a_p$ of
$f$ and the minimal polynomial $M(x)$ of the Hecke eigenvalue
$\lambda_p$ of $F$.  Then we look at the roots of $m$ and $M$ in
$\F_\ell$. The conjecture holds if for some choice of root of $m$
and some choice of root of $M$ the congruence holds.

The following statement summarizes our results on
Conjecture~\ref{conj:Harder}.
\begin{theorem}\label{thm:Harder}
  Let $r\leq 60$ be a multiple of $4$.  If $f\in S^{(1)}_r$ is a Hecke
  eigenform with coefficient field $\Q_f$ and $\ell$ is an ordinary
  prime in $\Q_f$ that divides the algebraic critical value
  $\widetilde{\Lambda}(f, r/2+2)$, then there exists a Hecke eigenform 
  $F\in S_{r/2,2}^{(2)}$ such that
  \begin{equation*}
    \mu_{p^\delta}(F)\equiv \mu_{p^\delta}(f) + p^{\delta(r/2+1)} +
    p^{\delta(r/2-2)}\pmod{\ell}
  \end{equation*}
  for
  \begin{equation*}
    p^\delta\in\{2, 3, 4, 5, 7, 8, 9, 11, 13, 17, 19, 23, 25, 27, 29,
    31, 125\}.
  \end{equation*}
\end{theorem}
\begin{proof}
  For weights $r\leq 28$, we have verified that there are no ordinary
  primes $\ell$ dividing $\widetilde{\Lambda}(f, r/2+2)$, so the
  statement is vacuously true.

  For weights $32\leq r\leq 60$ and
  \begin{equation*}
    p^\delta\in\{2, 3, 4, 5, 7, 9, 11, 13, 17, 19, 23, 25, 29, 31\}
  \end{equation*}
  we have verified the congruence for all large primes $\ell$ dividing
  the algebraic critical value.  The results are listed in
  Table~\ref{tbl:Harder}.

  The remaining cases $p^\delta\in\{8, 27, 125\}$ follow from the rest
  by Proposition~\ref{prop:cubes}.
\end{proof}

\begin{table}[h]
\begin{tabular}{|r|r|r|r|r|r|r|} \hline
  $r$ & $t$ &
  large $\ell\mid\textrm{Norm}(\widetilde\Lambda(f,t))$ & 
  $(k, j)$ & $\dim S_{k,j}^{(2)}$\\\hline
 $32$ & $18$ & $211$            & $(16, 2)$ &  $2$ \\ \hline
 $36$ & $20$ & $269741$         & $(18, 2)$ &  $2$ \\ \hline
 $40$ & $22$ & $509$            & $(20, 2)$ &  $3$ \\
      &      & $1447$           &           &      \\ \hline
 $44$ & $24$ & $205157$         & $(22, 2)$ &  $5$ \\ \hline
 $48$ & $26$ & $168943$         & $(24, 2)$ &  $5$ \\ \hline
 $52$ & $28$ & $173$            & $(26, 2)$ &  $8$ \\
      &      & $929$            &           &      \\
      &      & $4261$           &           &      \\
      &      & * $434167$       &           &      \\ \hline
 $56$ & $30$ & $173$            & $(28, 2)$ & $10$ \\
      &      & $1721$           &           &      \\
      &      & $38053$          &           &      \\
      &      & $1547453$        &           &      \\ \hline
 $60$ & $32$ & * $325187$       & $(30, 2)$ & $11$ \\ 
      &      & * $32210303$     &           &      \\
      &      & * $427092920047$ &           &      \\ \hline
\end{tabular}
\caption{A summary of the cases verified numerically for the proof of
  Theorem~\ref{thm:Harder}.  (The primes $\ell$ marked with a *
  have not been checked to be ordinary.)}
\label{tbl:Harder}
\end{table}

\subsection{Symmetric square $L$-functions of modular forms}
It is possible to associate higher-degree $L$-functions to modular forms,
by using various tensorial constructions.  We describe the $L$-function
attached to the symmetric square of a modular form.

Fix a Hecke eigenform $f\in S_r^{(1)}$ and 
let $\alpha_p,\beta_p$ be the roots of the polynomial
$X^2-a_pX+p^{r-1}$.  The associated \emph{symmetric square
$L$-function} is 
\begin{equation*}
L(\Sym^2 f,s)=\prod_p
\left((1-\alpha_p^2p^{-s})(1-\beta_p^2p^{-s})(1-\alpha_p\beta_pp^{-s})\right)^{-1}.
\end{equation*}
We take as factor at infinity
\begin{equation*}
  L_\infty(\Sym^2 f,s)=\frac{\Gamma(s)}{(2\pi)^s}
  \frac{\Gamma((s+2-r)/2)}{\pi^{(s+2-r)/2}}
\end{equation*}
and set
\begin{equation*}
  \Lambda(\Sym^2 f,s)=L_\infty(\Sym^2 f, s)L(\Sym^2 f, s).
\end{equation*}
Then $\Lambda(\Sym^2 f, s)$ has holomorphic continuation to $\C$ and
satisfies the functional equation
\begin{equation*}
  \Lambda(\Sym^2 f, s)=\Lambda(\Sym^2 f,2r-1-s).
\end{equation*}
We define the \emph{algebraic critical values}
\begin{equation}\label{eqn:sym}
  \widetilde{\Lambda}(\Sym^2 f, t)=
  \frac{L(\Sym^2 f, t)}{\pi^{2t-r+1}\langle f,f\rangle}
  \qquad\text{for }t=r,r+2,\ldots,2r-2,
\end{equation}
where $\langle \cdot,\cdot \rangle$ denotes the Petersson inner product.
(It is possible
  to express the algebraic critical values as quotients of the completed
  $L$-function $\Lambda(\Sym^2 f, t)$, which would be closer to the
  treatment of the usual $L$-function as given in the previous section.
  We prefer to take a quotient of $L(\Sym^2 f, t)$ instead, as this is
  the definition used in much of the existing work on symmetric square
$L$-values.)

\begin{theorem}[Zagier~\cite{Zagier}]
  If $f\in S_r^{(1)}$ is an eigenform and $t$ is even such that $r\leq
  t\leq 2r-2$, then $\widetilde{\Lambda}(\Sym^2 f, t)$ is an algebraic
  number.
\end{theorem}

Moreover, and this is useful for our computations, it can be shown
that $\widetilde{\Lambda}(\Sym^2 f, t)\in\Q_f$, see \cite{Shimura}.

\begin{conjecture}[Bergstr\"om-Faber-van der Geer-Harder]\label{conj:Sym}
  Let $f\in S_r^{(1)}$ be a Hecke eigenform with coefficient field $\Q_f$
  and let $\ell$ be a large prime in $\Q_f$.  Suppose $s\in\mathbb{N}$
  is such that $\ell^s$ divides the algebraic critical value 
  $\widetilde\Lambda(\Sym^2 f, t)$.
  Then there exists a Hecke eigenform $F\in S_{k, j}^{(2)}$, where
  $k=t-r+2$, $j=2r-t-2$, such that
  \begin{equation*}
    \mu_{p^\delta}(F)\equiv \mu_{p^\delta}(f)(p^{\delta(k-2)} + 1) \pmod{\ell^s}
  \end{equation*}
  for all prime powers $p^\delta$.
\end{conjecture}

The case $j=0$ concerns scalar-valued Siegel modular forms.  The first
examples of such congruences were found by Kurokawa~\cite{Kurokawa2}, who
conjectured that they should be governed by certain primes dividing the
numerators of algebraic critical values.  Kurokawa's conjecture was
recently proved by Katsurada and Mizumoto, who even extended these results 
to the case of scalar-valued Siegel modular forms of arbitrary degree
(see~\cite[Theorem 3.1]{KatsuradaMizumoto}).

In the vector-valued setting, the congruence in Conjecture~\ref{conj:Sym} 
was proved for the six
rational eigenforms of degree one (weights $12$, $16$, $18$, $20$, $22$,
$26$) by Dummigan in~\cite[Proposition 4.4]{Dummigan2}.  (Dummigan has
indicated that it should be possible to extend his Proposition~4.4 to 
higher weights, using a pullback formula as in Katsurada and
Mizumoto~\cite{KatsuradaMizumoto}.)
\begin{remark}
  The conjecture does not specify what is meant by a \emph{large} prime
  $\ell$.  Dummigan's result uses $\ell>2r$.  In the cases we have
  verified (see Theorem~\ref{thm:Sym} for details), it was sufficient to
  take $\ell>2$.
\end{remark}

Our numerical verification of Conjecture~\ref{conj:Sym} follows the
approach of the last section.  We highlight only the essential differences.

\subsubsection{Computation of the congruence primes $\ell$}
We find the appropriate primes $\ell$ by computing the algebraic
critical values directly from~\eqref{eqn:sym}.  The squared-norm
$\langle f, f\rangle$ of $f$ can be obtained from the identity
\begin{equation*}
  \langle f, f\rangle = \frac{(r-1)!}{2^{2r-1}\pi^{r+1}}L(\Sym^2 f, r),
\end{equation*}
so all we require is high-precision evaluation of the symmetric square
$L$-function at various points.  For this we use Dokchitser's
$L$-function calculator~\cite{Dokchitser} as wrapped in Sage, as well as
some Sage code made available to us by Martin Raum.

Having obtained a sufficiently precise floating point approximation to
the algebraic number $\widetilde{\Lambda}(\Sym^2 f, t)$, we then find
its minimal polynomial.  The congruence primes $\ell$ are the
primes larger than $2$ occurring in the factorization of the numerator 
of the norm of $\widetilde{\Lambda}(\Sym^2 f, t)$.

The critical values we obtain in this way agree with the ones computed
by Dummigan in the case of rational eigenforms\footnote{Dummigan
confirmed that a few of the factorizations from Table~1
in~\cite{Dummigan1} are incorrect.  Here are the values in question,
with their corrected factorizations:
\begin{align*}
  k&=16,r=3: & 2^{20}/3^7\cdot 5^3\cdot 7\cdot 11\cdot 13^2\cdot 17\\
  k&=16,r=11: & 2^{24}\cdot 839 / 3^{12}\cdot 5^8\cdot 7^4\cdot
  11^2\cdot 13^2\cdot 17\cdot 19\cdot 23\\
  k&=20,r=11: & 2^{27}\cdot 304477 / 3^{19}\cdot 5^8\cdot 7^4\cdot
  11^2\cdot 13^2\cdot 17^2\cdot 19\cdot 23\cdot 29
\end{align*} 
}, see 
Table~1 in~\cite{Dummigan1}.  We were also
able to verify the case $r=24$ by comparing our result with the trace of
$\widetilde{\Lambda}(\Sym^2 f, 46)$ as obtained (by theoretical means)
by Lanphier in~\cite{Lanphier}.

The following statement summarizes our results on
Conjecture~\ref{conj:Sym}.
\begin{theorem}\label{thm:Sym}
  Let $r\leq 32$.  If $f\in S^{(1)}_r$ is a Hecke
  eigenform with coefficient field $\Q_f$ and $\ell>2$ is a 
  prime in $\Q_f$ that divides the algebraic critical value
  $\widetilde{\Lambda}(\Sym^2 f, 2r-4)$, then there exists a Hecke 
  eigenform $F\in S_{r-2,2}^{(2)}$ such that
  \begin{equation*}
    \mu_{p^\delta}(F)\equiv \mu_{p^\delta}(f)(p^{\delta(r-4)} + 1)
    \pmod{\ell}
  \end{equation*}
  for
  \begin{equation*}
    p^\delta\in\{2, 3, 4, 5, 7, 8, 9, 11, 13, 17, 19, 23, 25, 27, 29,
    31, 125\}.
  \end{equation*}
\end{theorem}
\begin{proof}
  For weight $r=12$, we computed the numerator of the rational number
  $\widetilde{\Lambda}(\Sym^2 \Delta, 20)$ and found it to be
  $-2^{23}$, so there are no large primes dividing this algebraic
  critical value and the statement is vacuously true.

  For weights $16\leq r\leq 32$ and
  \begin{equation*}
    p^\delta\in\{2, 3, 4, 5, 7, 9, 11, 13, 17, 19, 23, 25, 29, 31\}
  \end{equation*}
  we have verified the congruence for all large primes $\ell$ dividing
  the algebraic critical value.  The results are listed in
  Table~\ref{tbl:Sym}.

  The remaining cases $p^\delta\in\{8, 27, 125\}$ follow from the rest
  by Proposition~\ref{prop:cubes}.
\end{proof}

\begin{table}[h]
\begin{tabular}{|r|r|r|r|r|r|r|r|} \hline
  $r$ & $t$ & 
  odd $\ell\mid\textrm{Norm}(\widetilde\Lambda(\Sym^2 f,t))$ & 
  $(k, j)$ & $\dim S_{k,j}^{(2)}$\\\hline
$16$ & $28$ & $373$ & $(14,2)$ & $1$ \\ \hline
$18$ & $32$ & $541$ & $(16,2)$ & $2$ \\
     &      & $2879$ & & \\ \hline
$20$ & $36$ & $439367$ & $(18,2)$ & $2$ \\ \hline
$22$ & $40$ & $281$ & $(20,2)$ & $3$ \\
     &      & $286397$ & & \\ \hline
$24$ & $44$ & $2795437$ & $(22,2)$ & $5$ \\
     &      & $256021114049$ & &\\ \hline
$26$ & $48$ & $4598642018203$ & $(24,2)$ & $5$ \\ \hline
$28$ & $52$ & $4017569791$ & $(26,2)$ & $8$ \\
     &      & $65593901428085768723$ & & \\ \hline
$30$ & $56$ & $937481$ & $(28,2)$ & $10$ \\
     &      & $4302719815755987715030485446839$ & & \\ \hline
$32$ & $60$ & $350747$ & $(30,2)$ & $11$ \\ 
     &      & $45130901953$ & & \\
     &      & $432796809552670722149$ & & \\ \hline
\end{tabular}
\caption{A summary of the cases
  verified numerically for the proof of Theorem~\ref{thm:Sym}.}
\label{tbl:Sym}
\end{table}


\subsection{Reduction of cubes to primes and squares of primes}
\label{sect:cubes}
We describe some elementary considerations that allow reducing the case
$\delta=3$ of both conjectures to the cases $\delta=1$ and $\delta=2$.
For ease of notation in this section, we will write
\begin{align*}
  g_\delta &= \alpha_0^\delta + \alpha_0^\delta \alpha_1^\delta\\
  G_\delta &= \alpha_0^\delta + \alpha_0^\delta \alpha_1^\delta +
  \alpha_0^\delta \alpha_2^\delta + 
  \alpha_0^\delta \alpha_1^\delta \alpha_2^\delta,
\end{align*}
where in the first line $\alpha_0$ and $\alpha_1$ are the Satake
parameters of $f\in S_r^{(1)}$, while in the second line $\alpha_0$,
$\alpha_1$ and $\alpha_2$ are the Satake parameters of 
$F\in S_{k,j}^{(2)}$.

In the degree one setting, we have the relation
$\alpha_0^2\alpha_1=p^{r-1}$, which allows us to express $g_2$ and $g_3$
in terms of $g_1$:
\begin{equation}\label{eqn:g2g3}
  g_2 = g_1^2-2p^{r-1},\qquad g_3 = g_1(g_1^2-3p^{r-1}).
\end{equation}

In the degree two setting, we have the relation
$\alpha_0^2\alpha_1\alpha_2=p^{2k+j-3}$, which allows us to express
$G_3$ in terms of $G_1$ and $G_2$:
\begin{equation}\label{eqn:G3}
  G_3 = \frac{1}{2}G_1\left(-G_1^2+3G_2+6p^{2k+j-3}\right).
\end{equation}

\begin{proposition}\label{prop:cubes}
  In Conjecture~\ref{conj:Harder} and Conjecture~\ref{conj:Sym}, the
  congruences for the case $\delta=3$ follow from the congruences for
  the cases $\delta=1$ and $\delta=2$.
\end{proposition}
\begin{proof}
  \
  \begin{enumerate}
    \item Define $h_\delta=p^{\delta(k+j-1)}+p^{\delta(k-2)}$ for all
      $\delta\geq 1$.  Then the congruence in
      Conjecture~\ref{conj:Harder} can be written
      \begin{equation*}
        (C_\delta):\qquad G_\delta\equiv g_\delta + h_\delta\pmod{\ell^s}.
      \end{equation*}
      It is easily seen that 
      \begin{equation}\label{eqn:h2h3}
        h_2=h_1^2-2p^{2k+j-3},\qquad h_3=h_1(h_1^2-3p^{2k+j-3}).
      \end{equation}
      (Observe the similarities between these equations
      and~\eqref{eqn:g2g3}.)

      We assume that the congruences $(C_1)$ and $(C_2)$ hold.  Using 
      Equations~\eqref{eqn:G3}, ~\eqref{eqn:g2g3} and~\eqref{eqn:h2h3}
      (in this order), we compute
      \begin{align*}
        G_3 &\equiv
        \frac{1}{2}(g_1+h_1)\left(-(g_1+h_1)^2+3(g_2+h_2)+6p^{2k+j-3}\right)\\
        &=\frac{1}{2}(g_1+h_1)\left((-g_1^2+3g_2)-2g_1h_1+(-h_1^2+3h_2+6p^{2k+j-3})\right)\\
        &=(g_1+h_1)\left(g_1^2-3p^{r-1}-g_1h_1+h_1^2\right)\\
        &=g_1^3-3p^{r-1}g_1+h_1^3-3p^{r-1}h_1\\
        &=g_3+h_3,
      \end{align*}
      after noting that, under the conditions of
      Conjecture~\ref{conj:Harder}, the weight parameters $r$, $k$ and
      $j$ are related by $r-1=2k+j-3$.
    \item The calculation is similar to the previous part.
      We let $h_\delta=p^{\delta(k-2)}+1$ for all 
      $\delta\geq 1$.  The congruence in Conjecture~\ref{conj:Sym} takes 
      the form
      \begin{equation*}
        (C_\delta^\prime):\qquad 
        G_\delta\equiv g_\delta h_\delta\pmod{\ell^2}.
      \end{equation*}
      We easily see that
      \begin{equation*}
        h_2=h_1^2-2p^{k-2},\qquad h_3=h_1(h_1^2-3p^{k-2}).
      \end{equation*}

      Assuming that congruences $(C_1^\prime)$ and $(C_2^\prime)$ hold,
      we obtain
      \begin{align*}
        G_3 &\equiv\frac{1}{2}g_1h_1
        \left(-g_1^2h_1^2+3(g_1^2-2p^{r-1})(h_1^2-2p^{k-2})+6p^{2k+j-3}\right)\\
        &=g_1h_1\left(g_1^2h_1^2-3p^{k-2}g_1^2-3p^{r-1}h_1^2+9p^{2k+j-3}\right)\\
        &=g_1h_1\left(g_1^2-3p^{r-1}\right)\left(h_1^2-3p^{k-2}\right)\\
        &=g_3h_3,
      \end{align*}
      where we used the relation $r=k+j$, valid under the conditions of
      Conjecture~\ref{conj:Sym}.
  \end{enumerate}
\end{proof}

\printbibliography

\end{document}